\documentclass[12pt]{amsart}

\usepackage{anysize}
\marginsize{1in}{1in}{1in}{1in}

\usepackage{amssymb}
\usepackage{amsbsy}
\usepackage{amscd}
\usepackage[mathscr]{eucal}
\usepackage{verbatim}
\usepackage{version}
\usepackage{graphicx}
\usepackage{color}

\definecolor{darkgreen}{rgb}{0.0, 0.7, 0.0}

\newenvironment{LG}{\noindent \color{darkgreen}{\bf LG:} \footnotesize}{}
\newenvironment{VS}{\noindent \color{blue}{\bf VS:} \footnotesize}{}
\newenvironment{REF}{\noindent \color{red}{\bf Referee:}\footnotesize}{}

\excludeversion{NB}
\excludeversion{LG}
\excludeversion{VS}
\excludeversion{REF}

\newenvironment{pf}{\proof[\proofname]}{\endproof}
\newenvironment{pf*}[1]{\proof[#1]}{\endproof}

\newtheorem{thm}{Theorem}%[section]

\newtheorem{lem}[thm]{Lemma}
\newtheorem{cor}[thm]{Corollary}

\theoremstyle{definition}

\newtheorem{defn}[thm]{Definition}
\newtheorem{rem}[thm]{Remark}

\newcommand{\thmref}[1]{Theorem~\ref{#1}}

\newcommand{\corref}[1]{Corollary~\ref{#1}}

\newcommand{\proj}{{\mathbb P}}
\newcommand{\CP}{\proj}

\newcommand{\ve}{\varepsilon}

\newcommand{\res}{\mathop{\text{\rm res}}}

\newcommand{\Coeff}{\mathop{\text{\rm Coeff}}}

\newcommand{\Res}{\operatornamewithlimits{Res}}

\newcommand{\sgn}{\mathop{\text{\rm sgn}}}

\newcommand{\A}{\mathbb{A}}

\newcommand{\C}{\mathbb{C}}

\renewcommand{\H}{\mathbb{H}}

\newcommand{\K}{\mathbb{K}}
\renewcommand{\L}{\mathbb{L}}

\renewcommand{\P}{\mathbb{P}}
\newcommand{\Q}{\mathbb{Q}}
\newcommand{\R}{\mathbb{R}}

\newcommand{\X}{\mathbb{X}}

\newcommand{\Z}{\mathbb{Z}}

\newcommand{\cC}{\mathcal{C}}

\newcommand{\mM}{\mathcal{M}}

\newcommand{\oO}{\mathcal{O}}

\def\<{\langle}
\def\>{\rangle}

\def\Syst{{\mathcal P}}

\begin{document}
\title[$\chi_y$-genera of relative Hilbert schemes]{The $\chi_{-y}$-genera of relative Hilbert schemes \\ for linear systems on Abelian and K3 surfaces}

\author{Lothar G\"ottsche}
\author{Vivek Shende}

\begin{abstract}
For an ample line bundle on an Abelian or K3 surface, minimal with respect to the polarization, the relative Hilbert scheme of points on the complete
linear system is known to be smooth.  We give an explicit expression in quasi-Jacobi forms for the $\chi_{-y}$ genus of the restriction of the Hilbert scheme to 
a general linear subsystem.
This generalizes a result of Yoshioka and Kawai for the complete linear system on the K3 surface, a result of Maulik, Pandharipande, and Thomas
on the Euler characteristics of linear subsystems on the K3 surface, and  a conjecture of the authors.   
\end{abstract}

\maketitle

\vspace{2mm}

%\today 
\section{Introduction}

Let $S$ be a smooth complex algebraic surface, and $L$ a line bundle over $S$.  Consider a linear system $\P^\delta \subset |L|$.  Let
$\cC \to \P^\delta$ the universal family of curves over the linear system, and $\cC^{[n]}$ the relative Hilbert schemes of points
on the fibres.  Under suitable hypotheses the Euler numbers of the $\cC^{[n]}$ control the number
$\delta$-nodal curves in $\P^\delta$ \cite{KST, KS}.  In these cases, the relative Hilbert schemes can be identified with the surface variant of the stable
pairs spaces of Pandharipande and Thomas \cite{PT1, PT3, KT1, KT2}.  When the spaces $\cC^{[n]}$ are smooth, their Euler numbers
may be computed by integrating Chern classes.  Writing these as ``tautological'' integrals over $S^{[n]}$ allows the fact \cite{EGL} that
all such integrals are determined by the Chern classes of $S$ and $L$ to be imported into enumerative geometry; one concludes
that the number of $\delta$-nodal curves in a $\delta$-dimensional slice of a $\delta$-very-ample linear system is given by a universal
formula.  This result had been previously conjectured by the first author \cite{G} and previously proven by  other methods \cite{Tz}. 

When $S$ is a K3 surface, explicit formulas for the $\chi(\cC^{[n]})$ are known.  The derivation of these
is however rather indirect: one shows \cite{MPT} an equivalence between the stable pairs and Gromov-Witten theories,\footnote{According
to \cite{MPT}, this is in the spirit of but does not directly follow from the conjectural 3-fold equivalence of \cite{MNOP}.}
and calculates \cite{BL} the latter.  Similar methods may be expected to work for the Abelian surface; alternatively, the formula
for the K3 surface determines the formula for the Abelian surface ``by universality''.  

In \cite{GS}, we replace the topological Euler characteristic $\chi$ with the Hirzebruch $\chi_{-y}$ genus ($\chi = \chi_{-1}$).
Imitating the formulas of \cite{KST} leads to 
putative refined curve counts which {\em conjecturally} 
are given by 
a universal formula in the Chern numbers.  The refinement recovers at $y = 1$ the counts
of complex curves, and conjecturally for a toric surface computes tropical refined Severi degrees \cite{BG, IM}.  The tropical refined
Severi degrees are defined combinatorially, 
but carry two meaningful enumerative specializations: at $y=1$ they count complex curves, and at $y=-1$ they count real curves
\cite{BG,M}. 

We moreover conjectured in \cite{GS}
an explicit formula for the refined invariants in the case of K3 or Abelian surfaces.  
Our goal here is to give a derivation of this formula, which in turn determines two of the four series involved
in the (still conjectural) formula for a general surface.

The quantity $$X_{-y}(x):=\frac{x(1-ye^{x(y^{1/2}-y^{-1/2})})}{y^{1/2}(1-e^{x(y^{1/2}-y^{-1/2})})}$$ is 
the normalized power series that defines the genus $\overline \chi_{-y}(M)=y^{-\dim(M)/2}\chi_{-y}(M)$. 
That is, for a vector bundle $E$ with Chern roots $e_1,\ldots,e_n$ we define $X_{-y}(E)=\prod_{i=1}^n X_{-y}(e_i)$. 
For a smooth projective variety $M$, we write $X_{-y}(M):=X_{-y}(T_M)$, and by the Riemann-Roch formula \cite{Hi} we have
$$\overline \chi_{-y}(M):=\int_M  X_{-y}(M) = y^{-\dim M/2} \sum_{p,q}(-1)^{p+q}y^q \mathrm{h}^{p,q}(M).$$

We collect the $\chi_{-y}$ genera of relative Hilbert schemes over
complete linear systems on the Abelian and K3 surfaces into generating series. 

\begin{defn}
Throughout we write $L_g$ to indicate a line bundle with no higher cohomology whose
sections have arithmetic genus $g$.  Note for all $g \ge 2$ there is an Abelian surface $A_g$ carrying a line bundle
$L_g$ such that the relative Hilbert schemes
$\cC_g^{[n]} \to |L_g|$ are smooth.  We define
\[
\A := \sum_{g \ge 2} \sum_{n\ge 0} \overline \chi_{-y}(\cC_g^{[n]}) t^{n+1-g} q^{g-1}.
\]
Similarly let $L_g$ be a linear system of genus $g$ curves on a K3 surface $K_g$
such that the relative Hilbert schemes $\cC_g^{[n]} \to |L_g|$ are smooth. We define
\[\K := \sum_{g\ge 0} \sum_{n\ge 0} \overline \chi_{-y}(\cC_g^{[n]})t^{n+1-g}q^{g-1}.\]
\end{defn}

We require two more generating series which contain the same information as $\A$. Writing $D = q \frac{d}{dq}$, we define

\begin{eqnarray*}
 \H & :=D^{-1}\A & =\sum_{g \ge 2} \sum_{n\ge 0} \overline \chi_{-y}(\cC_g^{[n]}) t^{n+1-g} \frac{q^{g-1}}{g-1}, \\
\X & := \frac{\H}{X_{-y}(\H)} & = \frac{y^{1/2}(1-e^{ (y^{1/2}-y^{-1/2}) \H})}{1-ye^{ (y^{1/2}-y^{-1/2})\H}}.
\end{eqnarray*}

As we recall in Section \ref{sec:universal}, when the surface $S$, line bundle $L$, and linear system $\P^\delta \subset |L|$ are 
such that $L$ has no higher cohomology and the relative Hilbert scheme $\cC^{[n]} \to \P^\delta$ has nonsingular total space, 
the Hirzebruch genus $\overline \chi_y(\cC^{[n]})$ is given by some universal expression (depending on $n, \delta$) in the
Chern classes of $S, L$.  Thus we may write $\overline \chi_{-y}({\cC}_{[S,L],\delta}^{[n]})$ for the evaluation
of this expression for any $S, L$, or indeed any specification of the Chern numbers 
$c_1(S)^2, c_1(S).L, L^2, c_2(S)$. We write $\overline \chi_{-y}({\cC}_{[S,L]}^{[n]}):=\overline \chi_{-y}({\cC}_{[S,L],\chi(L)-1}^{[n]})$ corresponding to the complete linear system. In speaking of $\chi(L), \chi(\mathcal{O}_S), g(L)$, etc., we mean the evaluation
on the specified Chern numbers of the formulas which usually give these quantities. 
More generally in the same way we may `integrate tautological classes over 
${\cC}_{[S,L],\delta}^{[n]}$.' 
Arguments similar to those of \cite{G, EGL, KST, GS} establish: 

\begin{thm} \label{thm:genf} 
There exist two more series $\mathbb{B}_1, \mathbb{B}_2\in \Q[y^{\pm 1/2},t^{-1}][[t,q]]$  such that the following hold:
%\begin{LG}
%I also think it would be useful (at least for me) to put the first few coefficents of 
%$B_1, B_2 $ in the paper (maybe not in the introduction, but the chapter where things are proved), to have a feel for the shape of things. But I first have to extract them from my computations.
%\end{LG}
\begin{REF} (1) In the first formula it shoud be $\overline \chi_y$ instead of $\chi_y$. \end{REF}
\begin{LG} I fixed it, I also changed $\chi_y$ to $\overline \chi_y$ a couple of times in the paragraph above to make it clearer.
\end{LG}
\begin{eqnarray*}
\X^{k}{\mathbb B}_1^{K_S^2} {\mathbb B}_2^{LK_S}
\K^{\chi(\oO_S)/2}{\mathbb A}^{1-\chi(\oO_S)/2}
& = & \sum_g \sum_{n\ge 0} t^{n-g+1} q^{g-1} \, \overline \chi_{-y}({\mathcal C}_{[S,L],\chi(L)-1-k}^{[n]}),\\
\H^{k}{\mathbb B}_1^{K_S^2} {\mathbb B}_2^{LK_S}
\K^{\chi(\oO_S)/2}{\mathbb A}^{1-\chi(\oO_S)/2} & = &
\sum_g \sum_{n\ge 0} t^{n-g+1} q^{g-1} \left(\int_{{\mathcal C}_{[S,L]}^{[n]}} X_{-y}({\mathcal C}_{[S,L]}^{[n]})\cap H^k\right).
\end{eqnarray*}

The meaning of the sum is that we fix $c_1(S)^2, c_2(S), c_1(S).L$, and vary only 
$L^2$, which we track by $g = g(L)$. 
\end{thm}
In both formulas, the summand on the RHS vanishes unless $g \ge k+2+LK_S-\chi(\oO_S)$.   Indeed, this may be checked when $L$ is an actual
line bundle with no higher cohomology on an actual surface $S$, where it amounts to $\dim |L|\ge k$.  

\vspace{2mm}

The Hodge polynomials of the relative Hilbert schemes on $K3$ surfaces were computed by Kawai and Yoshioka; specializing
these gives an explicit formula for $\K$.  In the present note we will compute $\A$.

We introduce some notation in order to state the answer.
Let $z$ be a complex variable and $\tau$ a variable from the complex upper half plane. We denote $y=e^{z}$, $q:=e^{2\pi i \tau}$. 
We denote one of the standard theta functions by
$$\theta(z)=\theta(z,\tau):=\sum_{n\in\Z} (-1)^n q^{\frac{1}{2}(n+\frac{1}{2})^2}y^{n+\frac{1}{2}}=q^{1/8}(y^{1/2}-y^{-1/2})\prod_{n>0} (1-q^n)(1-q^ny)(1-q^n/y),$$ and the Eisenstein series of weight $2$ by 
 $$G_2(\tau):=-\frac{1}{24}+\sum_{n>0} \left(\sum_{d|n} d\right)q^n.$$
 By abuse of notation we also write $\theta(y):=\theta(z)$, $G_2(q):=G_2(\tau)$.
Let $\vphantom{A}'$ denote $\frac{\partial}{\partial z}= y\frac{\partial}{\partial y}$, and $D=\frac{1}{2\pi i} \frac{\partial}{\partial \tau}=q\frac{d}{dq}$.
\begin{defn}
We write 
\begin{align*}
A(y,q):= \sum_{nd>0} \sgn(d) n^2 y^d q^{nd}.
\end{align*}
Here, as also a number of times below  $\sum_{nd>0}$ denotes the sum over pairs $n,d$ of integers with $nd>0$.
Thus  $A(y,q)$ can be viewed as a theta function for an indefinite lattice, as considered e.g. in 
\cite{Z}, \cite{GZ}.
\end{defn}

\begin{rem} 
$A(y,q)$ can be rewritten as follows.
\begin{align*}A(y,q)&=-\frac{1}{3}\frac{\theta'''(y)}{\theta(y)}-2G_2(q)\frac{\theta'(y)}{\theta(y)} \\
& = \frac{1}{\theta(y)}\left( -\frac{1}{6} D\theta'(y)-2G_2(q)\theta'(y)\right).
\end{align*} 
\end{rem}
\begin{pf}
The equality in the second line holds by the heat equation
$\theta''(y)=\frac{1}{2} D\theta(y)$.  
Now we prove the first equality: Denote $y_1=e^{z_1}$, $y=e^{z}$ for  complex variables $z_1,z$.
In \cite[page 456, compare (iii) and (vii)]{Z} it is proved that  
\begin{equation}\label{zagfun}
\frac{\theta'(0)\theta(y_1y)}{\theta(y_1)\theta(y)}=\frac{y_1y-1}{(y_1-1)(y-1)}-\sum_{nd>0}
\sgn(d)y_1^ny^dq^{nd}.
\end{equation}
We take the coefficient of $z_1^2$ of both sides of \eqref{zagfun}.
By the identity  \cite[eq.~(7)]{Z} we have
$$\frac{z_1\theta'(0)}{\theta(y_1)}=\exp\left(2\sum_{k\ge 2} G_k(\tau) \frac{z_1^k}{k!}\right).$$
This gives
$$\Coeff_{z_1^2}\left[\frac{\theta'(0)\theta(y_1y)}{\theta(y_1)\theta(y)}\right]=\Coeff_{z_1^3}\left[\frac{\theta(y_1y)}{\theta(y)}\right]+G_2(\tau)\Coeff_{z_1}\left[\frac{\theta(y_1y)}{\theta(y)}\right]
=\frac{1}{6}\frac{\theta'''(y)}{\theta(y)}+G_2(\tau)\frac{\theta'(y)}{\theta(y)}.$$
On the other hand 
$$\Coeff_{z_1^2}\left[\frac{y_1y-1}{(y_1-1)(y-1)}-\sum_{nd>0}
\sgn(d)y_1^ny^dq^{nd}\right]=-\frac{1}{2}\sum_{nd>0} \sgn(d)n^2y^{d}q^{nd}=-\frac{1}{2}A(y,d).$$
This proves the claim.\end{pf}
\begin{REF}(2) The referee cannot follow the proof for the left hand side. GIve a precise reference for the formula from [Z]
\end{REF}
\begin{LG} The computation as sketched was correct, but difficult to carry out. I now have a new proof, by making a Taylor development of $\theta(y)/\theta'(0)$.

Made reference to [Z] more precise.\end{LG}
We abbreviate
$$[n]_y:=\frac{y^{n/2}-y^{-n/2}}{y^{1/2}-y^{-1/2}}=y^{(n-1)/2}+y^{(n-3)/2}+\ldots+y^{-(n-1)/2}.$$

\begin{thm} \label{thm:A} 
Let $L_g$ be a linear system of genus $g$ curves on an Abelian surface $A_g$ such that the relative Hilbert schemes
$\cC_g^{[n]} \to |L_g|$ are smooth.  Then,
\begin{align*}
\mathbb{A} 
&:= \sum_{g \ge 2} \sum_{n\ge 0} \overline \chi_{-y}(\cC_g^{[n]}) t^{n+1-g} q^{g-1} \\
& =  (t+t^{-1}-y^{1/2}-y^{-1/2})\sum_{n>0,d>0}
n^2[d]_y[d]_{ty^{1/2}}[d]_{ty^{-1/2}} q^{nd} \\
& = \sum_{n>0,d>0} n^2\frac{y^{d/2}-y^{-d/2}}{y^{1/2}-y^{-1/2}} (t^d+t^{-d}-y^{d/2}-y^{-d/2})q^{nd} \\
& = \frac{1}{y^{1/2}-y^{-1/2}}\left(A\left(y^{1/2}t,q\right)+A\left(y^{1/2}/t,q\right)-A(y,q)  \right).
\end{align*}
\end{thm}

The following is a specialization (and slight reformulation) of a result of Kawai and Yoshioka \cite{KY}.

\begin{thm} \label{thm:K} \cite{KY} 
Let $L_g$ be a linear system of genus $g$ curves on a K3 surface $K_g$
such that the relative Hilbert schemes $\cC_g^{[n]} \to |L_g|$ are smooth.

\[
\K := \sum_{g\ge 0} \sum_{n\ge 0} \overline \chi_{-y}(\cC_g^{[n]})t^{n+1-g}q^{g-1} = 
%&=\frac{1}{(t^{1/2}y^{1/4}-t^{-1/2}y^{-1/4})(t^{1/2}y^{-1/4}-t^{-1/2}y^{1/4})}\cdot\\&
%frac{1}{q\prod_{n>0}(1-q^n\frac{t}{y^{1/2}})(1-q^n\frac{y^{1/2}}{t})
%(1-q^n ty^{1/2})(1-q^n \frac{1}{ty^{1/2}})(1-q^n\frac{1}{y})(1-q^ny)(1-q^n)^{18}}\\&
\frac{y^{-1/2}-y^{1/2}}{\Delta(q)}\frac{\theta'(0)^3}{\theta(y^{1/2}/t)\theta(ty^{1/2})\theta(y)}.
\]
\end{thm}

%\begin{LG} 
%This is checked. Check again.\end{LG}

Finally, we write explicitly the specialization ($y=1$) to Euler numbers.  
We denote $\overline B_1(q,t):={\mathbb B}_1(q,1,t)$, $\overline B_2(q,t):={\mathbb B}_2(q,1,t)$. 
In \cite{GS} we have introduced the function
$$\widetilde{DG_2}(y,q):=\sum_{n,d>0} n [d]_y^2 q^{nd}=\frac{D\log\frac{\theta'(0)}{\theta(y)}}{y-2+y^{-1}},$$
(the second identity is elementary).
From the third line in Theorem \ref{thm:A}, 
we have 
\begin{eqnarray*} 
{\mathbb A}(q,1,t) & = &(t-2+t^{-1})D\widetilde{DG_2}(t,q)=DD\log\frac{\theta'(0)}{\theta(t)}, \\
\H(q,1,t) & = & (t-2+t^{-1})\widetilde{DG_2}(t,q)=D\log\frac{\theta'(0)}{\theta(t)}.
\end{eqnarray*} 
From Theorem \ref{thm:K} it is easy to see that
\[
\K(q,1,t) =
\frac{\theta'(0)^2}{\Delta(q)\theta(t)^2} =
\frac{1}{(t-2+t^{-1})q\prod_{n>0} (1-q^n)^{20}(1-q^nt)^2(1-y^n/t)^2} =\frac{1}{\phi_{10,1}(t,q)}, \]
%\begin{LG}  checked \end{LG}
where $\phi_{10,1}(t,q)$ is up to normalization the unique Jacobi cusp form on $Sl(2,\Z)$ of weight $10$ and index $1$.
It is easy to see that  $X_{-1}(x)=(1+x)$, thus 
\[\X(q,1,t)=\frac{D\log\frac{\theta'(0)}{\theta(t)}}{1+D\log\frac{\theta'(0)}{\theta(t)}}.\] 
Putting this together we get the following.

\begin{cor}\label{cor:genf}
The generating series of integrals against  the hyperplane class is:
\begin{align*}
&\sum_{g\ge k+2+LK_S-\chi(\oO_S)}\sum_{n\ge 0} \left(\int_{{\mathcal C}_{[S,L]}^{[n]}} c({\mathcal C}_{[S,L]}^{[n]})\cap H^k\right)t^{n-g+1} q^{g-1}\\&=
\left(D\log\frac{\theta'(0)}{\theta(t)}\right)^k {\overline B}_1^{K_S^2} {\overline B}_2^{LK_S}\left(\frac{\theta'(0)^2}{\Delta(q)\theta(t)^2}\right)^{\chi(\oO_S)/2}\left(DD\log\frac{\theta'(0)}{\theta(t)}\right)^{1-\chi(\oO_S)/2}.
\end{align*}
The generating series of Euler characteristics is:
\begin{align*}
&\sum_{g\ge k+2+LK_S-\chi(\oO_S)}\sum_{n\ge 0} \chi({\mathcal C}_{[S,L],\chi(L)-1-k}^{[n]})t^{n-g+1} q^{g-1}\\&=
\left(\frac{D\log\frac{\theta'(0)}{\theta(t)}}{1+D\log\frac{\theta'(0)}{\theta(t)}}\right)^k {\overline B}_1^{K_S^2} {\overline B}_2^{LK_S}\left(\frac{\theta'(0)^2}{\Delta(q)\theta(t)^2}\right)^{\chi(\oO_S)/2}\left(DD\log\frac{\theta'(0)}{\theta(t)}\right)^{1-\chi(\oO_S)/2}.
\end{align*}
\end{cor}

As $D\log\frac{\theta'(0)}{\theta(t)}\in q\Q[t^{\pm 1}][[q]]$, we can develop
$$\left(\frac{D\log\frac{\theta'(0)}{\theta(t)}}{1+D\log\frac{\theta'(0)}{\theta(t)}}\right)^k =\sum_{m\ge 0} (-1)^{m}\binom{m+k-1}{k-1}\left(D\log\frac{\theta'(0)}{\theta(t)}\right)^{m+k},$$ thus for K3 surfaces the result specializes to formula (7) and Thm 6 of \cite{MPT}.

We return to the setting of \cite{GS}, where polynomials $N^{i}_{[S,L],\delta}(y)$ (called $\overline N^{i}_{[S,L],\delta}(y)$ there) were defined by the following formula, in which $g = g(L)$. 
\begin{equation} \label{eq:Ni} \sum_{n\ge 0} \overline \chi_{-y}(\cC_{[S,L],\delta}^{[n]}) t^{n+1-g}=\sum_{i=0}^\infty N^{i}_{[S,L],\delta}(y)
(t+t^{-1}-y^{1/2}-y^{-1/2})^{g-i-1}.\end{equation}
\begin{REF} (3) Should say they were called $\overline N^{i}_{[S,L],\delta}(y)$ in [GS], done above.\end{REF}
This formula refines the change of variable used to pass from Euler numbers of 
Hilbert schemes to enumerative information (of the sort sometimes 
called Gopakumar-Vafa or 'BPS' invariants).  In the good situation where $[S, L]$ 
comes from a line bundle on a surface with no higher cohomology and the appropriate
relative Hilbert schemes are nonsingular, $N^\delta|_{y=1}$ counts the number of 
$\delta$ nodal curves in a general $\P^\delta \subset |L|$ by \cite{KST}. 

According to \cite[Conj. 40]{GS}, we expect the vanishing $N^{i}_{[S,L],\delta}=0$ for $i>\delta$.  This was proven in case $K_S$ is numerically trivial.  Moreover, in this case,
or assuming the vanishing conjecture and in terms of two undetermined power series,  \cite[Conj. 67]{GS} gives a conjectural generating function for the highest order term $N^{\delta}_{[S,L],\delta}$. 
To establish this formula, and to better understand the $N^i$, it remains to develop 
the series introduced here in the variable $x=(t+t^{-1}-y^{1/2}-y^{-1/2})$.  
 
From \thmref{thm:genf} and Equation \eqref{eq:Ni} we obtain the generating series:

\begin{cor}\label{gsconj} Let $S$, $L$ be arbitrary, $g$ the arithmetic genus of $L$, then 
%$$\X^{k}{\mathbb B}_1^{K_S^2} {\mathbb B}_2^{LK_S}
%\K^{\chi(\oO_S)/2}{\mathbb A}^{1-\chi(\oO_S)/2}\in \Q[y^{\pm 1},x^{\pm 1},q^{-1}][[q]]$$
$$\sum_{i} N^i_{[S,L],\chi(L)-1-k}(y) x^{g-i-1}=\Coeff_{q^{g-1}}\, \big[\X^{k}{\mathbb B}_1^{K_S^2} {\mathbb B}_2^{LK_S}
\K^{\chi(\oO_S)/2}{\mathbb A}^{1-\chi(\oO_S)/2} \big].$$
\end{cor}

 We define
$$\widetilde\Delta(y,q):=\frac{\Delta(q)\theta(y)^2}{(y-2+y^{-1})\theta'(0)^2}=q\prod_{n>0} (1-q^n)^{20}(1-q^ny)^2(1-q^n/y)^2.$$

\begin{cor}\label{numconj}  \cite[Conj.~68]{GS}
\begin{REF} (4) There is no conjecture 66 in [GS]\end{REF}
\begin{LG} Corrected it is Conj. 68, .
\end{LG}
If $K_S$ is numerically trivial, then 
$$N_{[S,L],\chi(L)-1-k}^{\chi(L)-1-k}(y)=\Coeff_{q^{g-1}}\frac{\widetilde{DG}_2(y,q)^k \big(D\widetilde{DG}_2(y,q)\big)^{1-\chi(\oO_S)/2}}{\widetilde \Delta(y,q)^{\chi(\oO_S)/2}}.$$
\end{cor}

More generally, we want expressions for all the $N^i$, or in other words, we want to expand $\A$ and $\K$ in $x$ rather than $t$. 

We define polynomials $s_n(y)$ and their generating function $S(y,x)$ by
\begin{align*}s_n(y)&:=\sum_{k=0}^n\binom{n}{k}^2y^{k-n/2}=\Coeff_{t^n}\big(y^{-1/2}(1+t)(1+ty)\big)^n,\\
S(y,x)&:=\sum_{n\ge 0} (-1)^n \frac{s_n(y)}{(y^{1/2}-y^{-1/2})^{2n+1}} \frac{x^{n+1}}{n+1}.\end{align*}
Then we have 

\begin{thm}\label{higher}
\begin{eqnarray*}
\K & = & \frac{1}{\widetilde\Delta(y,q)}\left(\frac{1}{x}+\frac{1}{y^{1/2}-y^{-1/2}}\sum_{nd>0}\sgn(d)e^{S(y,x) (d-n)} y^dq^{nd}\right), \\
\A & = & \frac{1}{y^{1/2}-y^{-1/2}}\left(\sum_{nd>0}\sgn(d)(e^{S(y,x) d}-1) n^2(y^d-1)q^{nd}\right),\\
\H & = & \frac{1}{y^{1/2}-y^{-1/2}}\left(\sum_{nd>0}\sgn(d)(e^{S(y,x) d}-1) \frac{n}{d}(y^d-1)q^{nd}\right).
\end{eqnarray*} 
\end{thm}

To see explicitly the development of $\A, \K$ in $x$, we expand
\begin{equation}\label{Pgen}
e^{S(y,x)z}=:\sum_{n\ge 0} P_n(y,z) \frac{x^n}{(y^{1/2}-y^{-1/2})^n}\end{equation}
with $P_n(y,x)\in \Q[\frac{1}{y
-1},y,z],$ \begin{REF}(5) it should be $\Q[\frac{1}{y-1},y,z],$\end{REF}\begin{LG} corrected, although the polynomials are indeed in $\Q[\frac{1}{y
-1},z],$ as said before. This just says that the degree of the numerator in $y$ is not larger than that of the denominator.But this would need a little argument, and here seems to be no point.\end{LG}
\begin{REF}(5) The referee has different coefficients for $P_3$ and $P_4$. \end{REF}
\begin{LG} This is amazing actually, the referee was extremely careful. Yes, there was a misprint, the coefficients were swapped, corrected\end{LG}
e.g.
\begin{align*} P_0&=1,\ P_1=z, \ P_2=\frac{z^2}{2}-\frac{z}{2}\frac{y+1}{y-1},\ 
P_3=\frac{z^3}{6} -\frac{z^2}{2} \frac{y+1}{y-1}+
\frac{z}{3}\frac{y^2+4y+1}{(y-1)^2},\\
P_4&=\frac{z^4}{24}-\frac{z^3}{4}\frac{y+1}{y-1}+\frac{z^2}{24}\frac{11y^2+38y+11}{(y-1)^2}-
\frac{z}{4}\frac{y^3+9y^2+9y+1}{(y-1)^3}.
\end{align*}

\begin{rem} \label{K3Ab} It is remarkable that the generating functions for Abelian and K3 surfaces are determined by the same polynomials $P_i$: Using \corref{gsconj}.
we have on a  K3 surface 
$\sum_{g} N_{[S,L_g]}^{g}(y)q^{g-1}=\frac{1}{\widetilde\Delta(y,q)},$ and 
for $h\ge 1$:
\begin{equation} \label{KKgen}\sum_{g} N_{[S,L_g]}^{g-h}(y)q^{g-1}=\frac{1}{\widetilde\Delta(y,q)}\sum_{nd>0}\sgn(d)\frac{P_{h-1}(y,d-n)}{(y^{1/2}-y^{-1/2})^{h}}y^dq^{nd}.\end{equation}
On an Abelian surface we have for $h\ge 2$, 
\begin{equation}\label{AAgen} \sum_{g} N_{[A,L_g]}^{g-h}(y)q^{g-1}=\sum_{nd>0}\sgn(d)\frac{P_{h-1}(y,d)}{(y^{1/2}-y^{-1/2})^{h}} n^2(y^d-1)q^{nd}.\end{equation}

In fact, we first arrived at the formula asserted in Theorem \ref{thm:A} in the following manner.  The first author conjectured,
on the basis of numerical evidence, that Equations \eqref{KKgen}, \eqref{AAgen} held for some undetermined coefficients $P_i$.  
This suffices in principle to (conjecturally) determine $\A$ from $\K$.  Don Zagier made this determination explicit, providing
a formula for the $P_i$ and for $\A$.  Finally we have reversed the procedure, proving the formula for $\A$ geometrically and deriving
Equations \eqref{KKgen}, \eqref{AAgen} as consequences. 
\end{rem}

\vspace{2mm} {\bf Acknowledgements.}  We thank Don Zagier for the contributions mentioned immediately above,
and  K$\bar{\mathrm{o}}$ta Yoshioka for helpful correspondence about 
sheaves on Abelian surfaces. 
Part of this work was carried out while the first-named author was at the Max-Planck-Institut f\"ur Mathematik, Bonn. 

\section{Universality arguments} \label{sec:universal}

In this section we give the proof of Theorem \ref{thm:genf}. 

\vspace{2mm}  

\begin{defn}
Let $S$ be a surface,  $L$ a line bundle on $S$, and $L^{[n]}$ the corresponding tautological vector bundle on $S^{[n]}$.   
Let  $e^x$ denote a trivial line bundle with nontrivial $\C^*$ action with equivariant first Chern class $x$.  
Then we define\footnote{
Note this differs from \cite{GS} by the normalization by $y^{-\dim /2}$. }
\[
D^{S,L}(x,y,t):=\sum_{n\ge 0} t^n \int_{S^{[n]}} X_{-y}(TS^{[n]})\frac{c_n(L^{[n]}\otimes e^x)}{X_{-y}(L^{[n]}\otimes e^x)}.
\]
\end{defn}

As explained in \cite[Prop. 47]{GS}, for a linear subsystem $\P^\delta \subset |L|$ such that the relative Hilbert schemes $\cC^{[n]}_{[S,L],\delta} \to \P^\delta$
are all smooth -- e.g., a general $\delta$-dimensional linear subsystem when $L$ is  $\delta$-very-ample \cite{KST} -- 
we may extract the $\overline{\chi}_{-y}$ genera by taking a residue: 

\begin{equation}\label{DSLfor}
\sum_{n\ge 0} \overline \chi_{-y}({\cC}_{[S,L],\delta}^{[n]})t^n= \res_{x=0} \left[D^{S,L}(x,y,t)\left(\frac{X_{-y}(x)}{x}\right)^{\delta+1}\right]dx.\end{equation}
Since $D^{S,L}$ is defined by a tautological integral, by \cite{EGL} it depends only on the Chern numbers
$c_2(S), c_1(S)^2, c_1(S).c_1(L), c_1(L)^2$.  Thus we may make sense of it for arbitrary values of these
quantities.  Thus we view Equation \ref{DSLfor} as {\em defining} the quantities 
$\overline \chi_{-y}({\cC}_{[S,L],\delta}^{[n]})$ in terms of $\delta$, $n$, and the Chern numbers of $S, L$,
without any assumptions on even the existence of such a surface and line bundle.  

\begin{REF} (6)  misprint in next formula, corrected\end{REF}
The change of variable 
\[ 
q(x) = \frac{x}{X_{-y}(x)} = \frac{y^{1/2}(1-e^{x(y^{1/2}-y^{-1/2})})}{1-ye^{x(y^{1/2}-y^{-1/2})}}
\] 
is inverse to 
\begin{equation}\label{eq:xq}
x(q) = \frac{\log(1-y^{1/2}q)-\log(1-y^{-1/2}q)}{y^{-1/2}-y^{1/2}}=\sum_{n>0} [n]_y \frac{q^n}{n}.
\end{equation}
We find $\frac{dx}{dq}=\frac{1}{(1-y^{-1/2}q)(1-y^{1/2}q)}$.
Plugging into the residue formula \eqref{DSLfor},
and writing for convenience \[\overline{D}^{S,L}(q,y,t):= D^{S,L}(x(q),y,t),\]
we find
\begin{eqnarray*}
\sum_{n\ge 0} \overline \chi_{-y}({\cC}_{[S,L],\delta}^{[n]})t^{n} & = & 
\res_{q=0}\left[\overline D^{S,L}(q,y,t)q^{-(\delta+1)}\frac{1}{(1-y^{-1/2}q)(1-y^{1/2}q)}\right]\\ 
&= & \Coeff_{q^{\delta}}\left[\overline D^{S,L}(q,y,t)\frac{1}{(1-y^{-1/2}q)(1-y^{1/2}q)}\right].
\end{eqnarray*}
As the term in square brackets is a power series, we may re-sum to obtain
\begin{equation}
\label{DSLfor1}
\sum_{\delta\ge 0}\sum_{n\ge 0} \overline \chi_{-y}({\cC}_{[S,L],\delta}^{[n]})t^n q^{\delta}=\frac{\overline D^{S,L}(q,y,t)}{(1-y^{-1/2}q)(1-y^{1/2}q)}.
\end{equation}

Since $X_{-y}$ is a genus, by \cite{EGL} there exist power series $a_0,a_1,a_2,a_3\in \Q[y^{\pm  1/2}][[t,x]]$
such that $D^{S,L}(x,y,t)=a_0^{\chi(L)}a_1^{K_S^2}a_2^{LK_S} a_3^{\chi(\oO_S)}$ (for a detailed argument,
see \cite[Sec. 3.2]{GS}). Setting $$A_i(q,y,t) := a_i(x(q),y,t) \in \Q[y^{\pm 1/2}][[t,q]],$$ 
we get

\begin{equation}\sum_{\delta\ge 0}\sum_{n\ge 0} \overline \chi_{-y}({\cC}_{[S,L],_\delta}^{[n]})t^{n} q^{\delta}=
A_0^{\chi(L)}A_1^{K_S^2} A_2^{LK_S}A_3^{\chi(\oO_S)}\frac{1}{(1-y^{-1/2}q)(1-y^{1/2}q)}.
\end{equation}
Note that by \eqref{DSLfor1}, the coefficient of $q^0$ in $\overline D^{S,L}(x,y,t)$  
is $\big((1-y^{-1/2}t)(1-y^{1/2}t)\big)^{g-1}$ for $g$ the arithmetic genus of a curve in $|L|$.
Thus  $A_i(q,y,t)\in\big((1-y^{-1/2}t)(1-y^{1/2}t)\big)^{l_i}+q\Q[y^{\pm 1/2}][[t,q]]$ with $l_0=1,l_1=0,l_2=1,l_3=-1$.

If $R$ is a commutative ring, and $f\in R[[q]]$ is an invertible power series, we denote by 
$f^{-1}$ its compositional inverse.
Let
$$\X(q,y,t):=\left(\frac{qt}{A_0}\right)^{-1}\in \Q[y^{\pm 1/2},t^{-1}][[t,q]].$$
This is set up so that 
$$\frac{\X(q,y,t)}{A_0(\X, y, t)} = q/t,$$
and hence $A_0(\X, y, t)=q^{-1}t \X(q, y, t)$. 

Denoting ${\mathbb B}_1(q,y,t):=A_1(\X, y,t)$, ${\mathbb B}_2(q,y,t):=A_2(\X, y,t) q/t$, 
${\mathbb B}_3(q,y,t):=A_3(\X, y,t)t/q$, the substitution $q \mapsto \X$ gives: 

\begin{equation} \label{powd}
\sum_{\delta\ge 0}\sum_{n\ge 0} \overline \chi_{-y}({\cC}_{[S,L],\delta}^{[n]})t^{n+1-g} \X^{\delta}=
\frac{(\X/q )^{\chi(L)}{\mathbb B}_1^{K_S^2} ({\mathbb B}_2/q)^{LK_S}({\mathbb B}_3\cdot q)^{\chi(\oO_S)}}{(1-y^{1/2}\X)(1-y^{-1/2}\X)}.
\end{equation}

As in \cite{G} we use the residue formula. 
Let $R$ be a commutative ring, and $f\in R[[q]]$, $g\in aq+q^2R[[q]]$, with $a$ invertible in $R$, then
$$f=\sum_{k=0}^\infty g(q)^k \left.\left[\frac{f(q)Dg(q)}{g(q)^{k+1}}\right]\right|_{q=0}.$$
We apply this to Equation \eqref{powd} with $g(q)=\X$.  On the one hand,
$\sum_{n\ge 0}\chi_{-y}({\mathcal C}_{[S,L],\delta}^{[n]})t^{n+1-g}$ is the coefficient of $\X^{\delta}$ of the RHS.
On the other, taking the coefficient by the residue formula above gives, with $g$ again the arithmetic genus of a curve in $|L|$,
\begin{align*}
\sum_{n\ge 0}\overline \chi_{-y}({\mathcal C}_{[S,L],\delta}^{[n]})t^{n+1-g} 
&= \left.  \frac{D\X\cdot  \X^{-\delta - 1} \cdot (\X/q )^{\chi(L)}{\mathbb B}_1^{K_S^2} 
 ({\mathbb B}_2/q)^{LK_S}({\mathbb B}_3\cdot q)^{\chi(\oO_S)}}{(1-y^{1/2}\X)(1-y^{-1/2}\X)} %\right]
 \right|_{q=0} 
 \\ &=\mathrm{Coeff}_{q^{g-1}} 
\frac{D\X  \cdot (\X)^{\chi(L) - \delta - 1}{\mathbb B}_1^{K_S^2} 
 {\mathbb B}_2^{LK_S}{\mathbb B}_3^{\chi(\oO_S)}}{(1-y^{1/2}\X)(1-y^{-1/2}\X)}.
\end{align*}

We collect terms with fixed $k=\chi(L)-1-\delta$, i.e. (if we assume that $L$ has now higher cohomology) $k$ is the number of point conditions we impose to cut down to $\P^\delta$.  We now explicitly note the genus of the line bundle appearing in its subscript. \begin{REF} (7) Explain were the lower index in the formula below comes from, also it should by $LK_S$ instead of $-LK_S$.\end{REF}
\begin{LG} The referee is right about the sign: fixed. Added next sentence to explain lower index.\end{LG}
Note that we always have $\delta\ge 0$, which by definition of $k$ and $\chi(L)=g-LK_S-1+\chi(\oO_S)$ translates into $g\ge k+2+LK_S-\chi(\oO_S)$.

\begin{cor}\label{powl} Fix $k\ge 0$, then 
\begin{align*}
\sum_{g\ge k+2+LK_S-\chi(\oO_S)}\sum_{n\ge 0} \overline \chi_{-y}({\mathcal C}_{[S,L_g],\chi(L_g)-1-k}^{[n]})t^{n-g+1} q^{g-1}=\frac{
\X^{k}{\mathbb B}_1^{K_S^2} {\mathbb B}_2^{L_gK_S}{\mathbb B}_3^{\chi(\oO_S)} D\X}{(1-y^{-1/2}\X)(1-y^{1/2}\X)}.
\end{align*}
\end{cor}

In particular, when $S = A$ is an Abelian surface,
$$
\frac{D\X}{(1-y^{-1/2}\X)(1-y^{1/2}\X)}=
\sum_{g\ge 2}\sum_{n\ge 0}\overline  \chi_{-y}({\mathcal C}_{[A,L_g],\chi(L_g)-1}^{[n]})t^{n-g+1} q^{g-1}=:\mathbb{A}.$$
Note that
$$\frac{D\X}{(1-y^{-1/2}\X)(1-y^{1/2}\X)}=D\left(\frac{\log(1-y^{-1/2}\X)-
\log(1-y^{1/2}\X)}{y^{-1/2}-y^{1/2}}\right).$$
Thus, by $\H=D^{-1}\A$, we see
\begin{equation}\label{eq:HX}\H=\frac{\log(1-y^{1/2}\X)-
\log(1-y^{-1/2}\X)}{y^{-1/2}-y^{1/2}}=x(\X).\end{equation}
We have already seen  how to invert this function:
\begin{equation}
\X= q(\H)=\frac{\H}{X_{-y}(\H)} = \frac{y^{1/2}(1-e^{\H (y^{1/2}-y^{-1/2})})}{1-ye^{\H (y^{1/2}-y^{-1/2})}}.
\end{equation}

Similarly, when $S = K$ is a K3 surface,
$$
\mathbb{A} \mathbb{B}_3^2 =
\sum_{g\ge 0}\sum_{n\ge 0}\overline  \chi_{-y}({\mathcal C}_{[K,L_g],\chi(L_g)-1}^{[n]})t^{n-g+1} q^{g-1}=:\mathbb{K},$$
and so $\mathbb{B}_3 = (\K / \A)^{1/2}$. Putting everything together, this proves the first formula of \thmref{thm:genf}.

We now prove the second formula.  The argument takes place for fixed $[S, L]$.
Write $H$ for the pullback of the hyperplane class from $|L|$.
Denote $$Z_{[S,L]}(x,y,t):=t^{1-g}D^{S,L}(x,y,t) q(x)^{-\chi(L)}.$$
Equation \eqref{DSLfor}  asserts that when the relevant spaces are smooth, we have 
\begin{align*}
\sum_{n\ge 0} \overline \chi_{-y}({\cC}_{[S,L],\chi(L)-1-k}^{[n]})t^{n+1-g} &=\res_{x=0} \left[t^{1-g} D^{S,L}(x,y,t) q(x)^{-\delta-1}\right]dx\\&=\res_{x=0} \left[Z_{[S,L]}(x,y,t)  q(x)^{k}\right]dx.
\end{align*}
By the same proof, if the $\cC_{[S,L]}^{[n]}$ are smooth,  we have 
$$\sum_{n\ge 0} \left(\int_{{\mathcal C}_{[S,L]}^{[n]}} X_{-y}({\mathcal C}_{[S,L]}^{[n]})\cap H^k\right)t^{n+1-g}= \res_{x=0} \left[Z_{[S,L]}(x,y,t)x^k \right]dx.$$
Write $f(q,y,t):={\mathbb B}_1^{K_S^2} {\mathbb B}_2^{LK_S}\K^{\chi(\oO_S)/2}\A^{1-\chi(\oO_S)/2}$. 
We have shown
\[\res_{x=0} \left[Z_{[S,L]}(x,y,t)  q(x)^{k}\right]dx=\Coeff_{q^{g-1}}\left[\X^{k}f(q,y,t)\right].
\]
Let again $x(q)$ from \eqref{eq:xq} be the compositional inverse of $q(x)$. 
Write $x(q)^k:=\sum_{l\ge k} a_l(y) q^l$, such that $\sum_{l\ge k} a_l(y)q(x)^l=x^k.$
Thus we get
\begin{align*}\res_{x=0}& \left[Z_{[S,L]}(x,y,t)x^k \right]dx= \res_{x=0} \left[\sum_{l\ge k} a_l(y) Z_{[S,L]}(x,y,t)q(x)^l\right]dx\\&=\Coeff_{q^{g-1}}\left[\sum_{l\ge k} a_l(y)\X^{l}f(q,y,t)\right]=\Coeff_{q^{g-1}}\left[x(\X)^k f(q,y,t)\right]=\Coeff_{q^{g-1}}\left[\H^k f(q,y,t)\right].
\end{align*}
The last equality is by \eqref{eq:HX}.

\section{Calculations for the Abelian surface} \label{sec:Yoshioka}

Kawai and Yoshioka determined $\K$ by comparing
various moduli spaces of stable sheaves and stable pairs on a K3 surface
\cite{KY}.  A modification of their argument suffices to determine $\A$ except for the coefficient
of $t^0$, and 
a vanishing result in \cite{GS} allows us to determine this
coefficient from the rest. 

\subsection{Yoshioka's lemma}

\begin{lem} \cite[Lem. 2.1]{Y-reflections}  Let $X$ be a smooth projective surface with polarization $H$, and let $C$ be a curve class minimizing
$C.H$.  For a sheaf $F$ with $c_1(F) = dC$, we write
$\deg(F) = d = (c_1(F).H)/(C.H)$.  Let $(r, d)$ and $(r_1, d_1)$ be pairs of integers such that $r_1 d - d_1 r = 1$, with $r \ge 0$ and $r_1 > 0$.  Let
$(r_2, d_2) := (r, d) - (r_1, d_1)$.  Below let $E_i$ be of rank $r_i$ and degree $d_i$, and let $E_1$ always be a vector bundle. 
\begin{itemize} 
	\item If $E_1,$ $E_2$ are $\mu$-stable , then every nontrivial extension \[0 \to E_1 \to E \to E_2 \to 0\] is $\mu$-stable.
	\item If $E_1, E$ are $\mu$-stable, then for any vector
	subspace $V \subset \mathrm{Hom}(E_1, E)$, if the evaluation map $V \otimes E_1 \to E$ is not injective, then it is surjective in codimension $1$. 
	Moreover 
	\begin{itemize}
		\item If $V \otimes E_1 \to E$ is injective, then the cokernel is $\mu$-stable.
		\item if $V \otimes E_1 \to E$ is surjective in codimension 1, the kernel is $\mu$-stable. 
	\end{itemize} 
\end{itemize} 
\end{lem}

\begin{rem} In \cite{Y-reflections}, this lemma is proven under the assumption that $NS(X) = \Z$, but in \cite{KY} it is pointed out that the proof only
requires the assumption stated above.
\end{rem}

Note that if $d = 1$, i.e. we are looking at sheaves with $c_1 = C$, then the condition $r_1 d - d_1 r = 1$ is always satisfied by $(r_1, d_1) = (1,0)$, 
hence we may always take $E_1 = \oO_X$.  We now extract explicitly the special cases we will be concerned with.

\begin{cor}
	Let $X$ be a smooth projective surface with polarization $H$, and let $C$ be a curve class minimizing $C.H$.
   \begin{itemize}
   	\item Assume $F$ is $\mu$-stable and $c_1(F) = C$.  Then every nontrivial extension $0 \to \oO_X \to E \to F \to 0$ is $\mu$-stable. 
	\item Assume $E$ is $\mu$-stable of positive rank and $c_1(E) = C$.  Then any non-zero section induces an exact sequence 
	$0 \to \oO_X \to E \to F \to 0$ and $F$ is $\mu$-stable. 
   \end{itemize}
   \begin{REF} (8) Stable means $\mu$-stable.\end{REF}
\end{cor}
\begin{proof}
The only thing which is not immediate from the lemma is to check is the possibility 
in the second case that $\oO_X \to E$ is surjective in codimension $1$ rather than being injective.  But then in any case $E$ must
either be torsion (which it is not by assumption) or the map from $\oO_X$ must be an isomorphism in codimension $1$, in which
case the kernel must be a torsion subsheaf of $\oO_X$, hence zero. 
\end{proof}

Let $\mM(r,d,e)$ denote the moduli space of semistable sheaves of rank $r$, degree $d$, and Euler number $e$. \footnote{Note we are not
indexing by the Mukai vector.}  
We will below always assume that $\mM(r,d,e)$ only consists of $\mu$-stable sheaves. 
Let $\Syst^1(r,d,e)$ be the space of ``coherent systems'' \cite{LeP}, i.e. it parameterizes a stable sheaf (of rank $r$, degree $d$, 
and Euler number $e$) plus a section, up to isomorphism.  
This corresponds to a special choice of the stability condition for pairs, which ensures that a pair of sheaf and section is stable, if and only if the sheaf is stable. 
There is a forgetful map $\Syst^1(r,d,e) \to \mM(r,d,e)$ with fibre $\CP \mathrm{H}^0(E)$ over a sheaf $E$. 

The above corollary implies the existence of another map: 
 
 \begin{cor}
   For $r \ge 0$, there exists a morphism $\Syst^1(r+1,d,e + \chi(\oO_X)) \to \mM(r,d,e)$ which takes
   $\oO_X \to E$ to its cokernel.  The fibre over a sheaf $F$ is $\CP \mathrm{Ext}^1(F, \oO_X)$. 
 \end{cor}
 
 Let us consider the space $\Syst^1(0, 1, e)$.  This by definition consists of a stable, rank zero sheaf $E$ together with a section 
 $\oO_X \to E$.   By stability, 
$E$ is a pure sheaf supported on a curve (i.e. torsion free with rank one on its support).  As explained in \cite[Appendix B]{PT3}, dualizing gives an isomorphism between
$\Syst^1(0,1,e)$ and the relative Hilbert scheme  \begin{REF} which Hilbert scheme, what is the base, need the dimension later.\end{REF} of degree $e + g - 1$, 
where $g$ is the arithmetic genus of the support of $E$, over the moduli space $\mathcal M$ of curves of degree $dC$ on $S$.
\begin{LG} The last half sentence is added. \end{LG}

\subsection{A relation between moduli spaces}

We now specialize to $K_X = \oO_X$.  Note in this case that if $E$ is any stable sheaf with zero rank or positive first Chern class, then 
\[\mathrm{H}^2(E) = \mathrm{Hom}(E,K_X)^* = \mathrm{Hom}(E,\oO_X)^* = 0.\]  In the zero rank case the last equality is obvious; 
for positive rank it is ensured by stability.  Additionally we have $\mathrm{Ext}^1(F, \oO_X) = \mathrm{H}^1(F)^*$ by Serre duality.  Thus
the dimensions of the fibres of the two maps to $\mM(r,d,e)$ are related: 
\[\Syst^1(r,d,e) \xrightarrow{\P \mathrm{H}^0(F)} \mM(r,d,e) \xleftarrow{\P \mathrm{H}^1(F)} \Syst^1(r+1,d,e + \chi(\oO_X)). \] 
We indicate throughout the first map by $\Syst \to \mM$ and the second by 
$\mM \leftarrow \Syst$.  

We denote the Hodge polynomial of $V$ by $[V]$. 
We write $\L$ for the Hodge polynomial of the affine line, 
and $[n] = [\P^{n-1}]$ for $n \in \Z_{>0}$.  We also write
$[0] = 0$ and $[-n] = -\L^{-n} [n]$.  

Let $\mM(r,d,e)_s$ denote the locus with $h^0 = s$.  
Since the map $\Syst \to \mM$ is given on the above 
strata as the projectivization of a vector bundle, we have:
\[ [\Syst^1(r,d,e)] = \sum_i [e+ i] [\mM(r,d,e)_{e+i}] .\]
Considering instead the map $\mM \leftarrow \Syst$ and 
using the vanishing of $h^2(F)$ to write $\chi(F) = h^0(F) - h^1(F)$, we have: 
\[ [\Syst^1(r+1,d,e+\chi(\oO_X))]   =  \sum_i [i] [\mM(r,d,e)_{e+i}].  \]
As observed in \cite{KY}, this establishes a recursion:
\begin{eqnarray*}
[\Syst^1(r,d,e)] & = & \sum_i [e+ i] [\mM(r,d,e)_{e+i}]\\
& = & [e][\mM(r,d,e)] + \L^e \sum [i] [\mM(r,d,e)_{e+i}] \\
& = & [e][\mM(r,d,e)] +\L^e [\Syst^1(r+1,d,e+\chi(\oO_X))].
\end{eqnarray*}
Because the dimension of $\Syst$ contains a term $-re$, iterating this
leads to empty moduli spaces $\Syst$ when either (1)
we are working on the a K3 surface where $e$ is increased at each step
by $\chi(\oO_X) = 2$, or (2) when $e > 0$ and we are on the Abelian surface.\footnote{When $e = 0$ on 
the Abelian surface, we learn that $[\Syst^1(r, d, 0)]$ is independent of $r \ge 0$, but
we have not found any use for this fact.}
In these cases we may sum the recursion (which is to say, the following sum is really a finite sum):
\begin{equation} \label{eq:rec}
[\Syst^1(r,d,e)] = \sum_{b=0}^\infty [e+b\chi(\oO_X)]\L^{\sum_{j=0}^{b-1} e + j\chi(\oO_X)} [\mM(r+b, d, e + b \chi(\oO_X))].
\end{equation}

To evaluate this sum, It remains to (1) use the deformation equivalence of moduli of sheaves on K3 or Abelian surfaces and 
Hilbert schemes of points on these surfaces and then (2) plug in the formula for the Hodge polynomial of the Hilbert scheme \cite{G1}. 
For the K3 surfaces, this is done in \cite{KY}.  We proceed now to the case of the Abelian surface, where one must moreover
deal with the $e \le 0$ case in some other way.

\subsection{Abelian surfaces}

Let $A$ be an Abelian surface.  

We change notation slightly from the previous section, and write 
$\mM_{num}(r,d,e)$ for what was written there $\mM(r,d,e)$: the moduli 
space of sheaves where $c_1 = d$ is fixed only in cohomology. 
We now denote $\mM(r,C,e)$  the moduli space where $c_1 = C$ is fixed in
Pic, and similarly for the spaces $\Syst$.  Note the discussion there
 for $\mM_{num}, \Syst_{num}$ is equally valid for (what is here called) $\mM, \Syst$ 
 and thus Equation \ref{eq:rec} holds for these spaces as well.  

Twisting by line bundles gives an isomorphism $A^\vee/A^\vee [r] 
\times \mM(r, C, e) \cong \mM_{num}(r, d, e)$.
The space $\mM_{num}(r,d,e)$ has dimension
\[\dim \mathrm{Ext}^1(F, F) = 2 - \chi(F,F) = 2 - \int \mathrm{ch}(F)^\vee \mathrm{ch}(F) = 2 + C^2 - 2 r e = 2g(C) - 2 r e.\]
If this is greater than $2$, then according to \cite[Thm. 0.1]{Y-abelian}, \begin{REF} It should be $\mM_{num}$: fixed.\end{REF} $\mM_{num}(r,d,e)$ is deformation equivalent to 
$A^\vee \times A^{[n]}$ for the appropriate $n$.  

%To proceed we must now specialize
%$[ \cdot ]$ to mean Hodge polynomial.  Then 
For $(g-1) > re$ we have 
$[\mM_{num}(r,d,e)] = [A^\vee \times A^{[g-1-re]}]$ and 
$[\mM(r,C,e)] = [A^{[g-1-re]}]$.  In particular, in this case 
$\chi_{-y}(\mM(r,C,e)) = 0$. 
On the other hand, according to \cite[Lem. 4.19]{Y-abelian}, 
when $r | (g-1)$, then $\mM(r,C,\frac{g-1}{r} )$ is a finite set of 
$r^2$ points. 
Thus for $e > 0$, equation (\ref{eq:rec}) gives:
\[
[\Syst^1(0,C,e)] = [e] \sum_{b=0}^\infty \L^{b e} [\mM(b, C, e)]
= [e] \left(\sum_{b=0}^{b < (g-1)/e} \L^{be} [A^{[g-1-be]}] + \L^{g-1} \sum_{b = (g-1)/e} b^2
\right).
\]
Now we treat the case of negative Euler characteristic.
There are morphisms
\begin{align*}
\pi_+: \Syst^1(0,C,e)&\to \mM(0,C,e), (s,E)\mapsto E,\\
\pi_-:\Syst^1(0,C,-e)&\to \mM(0,C,e), (s,F)\mapsto {\mathcal E}xt^1_{\oO_S}(F,\omega_S)={\mathcal H}om(F,\omega_D),\end{align*}
where in the second line $D$ is the support curve of $F$. By Serre duality on the support curve we get
$[\pi_+^{-1}(E)]-\L^e[\pi_-^{-1}(E)]=[e]$ for all $E\in \mM(0,C,e)$. Furthermore $\mM(0,C,e)$ is stratified into locally closed subsets over which
$\pi_+$ and $\pi_-$ are projectivizations of vector bundles. This gives
\[[\Syst^1(0,C,e)] - \L^e [\Syst^1(0,C,-e)] = [e][\mM(0,C,e)]=0.\]  
\begin{REF} Referee does not understand how this should follow from the proposition\end{REF}
\begin{LG} Note: it is Prop 42. so that I fixed. On the other hand I think one should explain better how this follows from the Proposition. Indeed it looks more like a duality, and it feels more that it somehow follows from the proof of Prop 42, but maybe not directly from the statement.\end{LG}
\begin{VS} Let's just drop the citation to prop 42 there and give the direct
explanation instead: the two things being subtracted both map to the
thing in the middle; the first in the obvious way and the second by
$Ext^1_{on the surface}(\omega_{surface}, . )$ or, what is the same,
$Hom_{on the supporting curve}(\omega_{curve}, . )$.  The fibres are
projective spaces.  over any given point, if I write the same
expression for the difference in the fibres, I get $[e]$ by Serre
duality on the support curve (or I guess just as well on the surface).
Passing to hodge polynomial (or for that matter motive since these
projective bundles are in fact proj of vector bundles) gives the
stated equality.
\end{VS}

We now pass from $[\cdot]$ to $\overline{\chi}_{-y}$; note that in addition to specializing
parameters we must multiply by $y^{-\dim /2}$; by the isomorphism with Hilbert schemes,
we have $\dim \Syst^1(0,C,e) = 2(g-1) + e - 1$.  All terms containing
$[A^{[\ge 1]}]$ vanish, leaving:

\begin{eqnarray*}
\sum_{e \ne 0} t^{e} \overline{\chi}_{-y} (\Syst^1(0,C,e)) 
& = &  y^{1/2} \sum_{e | (g(C)-1)}    \frac{y^e-1}{y-1} (y^{-e/2} t^e + y^{e/2} y^{-e} t^{-e}) \left(\frac{g-1}{e}\right)^2 \\
& = &  \sum_{e | (g(C)-1)}   \frac{y^{e/2}-y^{-e/2}}{y^{1/2}-y^{-1/2}} (t^e + t^{-e}) \left(\frac{g-1}{e} \right)^2.
 \end{eqnarray*}

 It remains to determine the contribution of sheaves with $e=0$.  
 Let $g = g(C)$, and 
 \[f_g(y,t):=\sum_{e\in \Z} \overline \chi_{-y} (\Syst^1(0,C,e)) t^e=\sum_{e\in \Z} \overline \chi_{-y} ({\mathcal C}^{[e+g-1]}_{[A,C],g-2})t^e. \]
 \begin{REF} there was an extra $]$ in the formula: fixed\end{REF}
Thus by  \eqref{eq:Ni} the $N^i$ are defined by expanding
$$f_g(y,t)=\sum_{i \ge 0}N^i_{[A,C],g-2}(y)  \left(\frac{(1-y^{-1/2}t)(1-y^{1/2}t)}{t}\right)^{g-1-i}.$$
We have shown in \cite{GS} that, for surfaces with trivial canonical bundle,
$N^i_{[S,L],\delta}(y)=0$ for $i>\delta$. Applying this here, 
we see that $f_g(y,t)$ is a Laurent polynomial in $t,y^{1/2}$, divisible by 
$\frac{(1-y^{-1/2}t)(1-y^{1/2}t)}{t}$. In particular $f_g(y,y^{1/2})=0$, in other words
$$\Coeff_{t^0}(f_g(y,t))=-\sum_{e | (g(C)-1)}   \frac{y^{e/2}-y^{-e/2}}{y^{1/2}-y^{-1/2}} (y^{e/2} + y^{-e/2}) \left(\frac{g-1}{e} \right)^2,$$ and 
$$f_g(y,t)=\sum_{e | (g(C)-1)}   \frac{y^{e/2}-y^{-e/2}}{y^{1/2}-y^{-1/2}} (t^e + t^{-e}-y^{e/2}-y^{-e/2}) \left(\frac{g-1}{e} \right)^2.$$
 Putting $d:=e$, $n:=(g-1)/e$, we see that $\A=\sum_{g\ge 2} f_g(y,t)q^{g-1}$ is given by
 the second line  of \thmref{thm:A}.
 
 \begin{rem}
For the Hodge polynomial $\overline h(X)=(xy)^{-dim(X)/2} \sum_{p,q} \text{h}^{p,q}(X)(-x)^p(-y)^q$, the above argument gives \begin{align*}
\sum_{e \ne 0} t^{e} \overline h (\Syst^1(0,C,e)) 
& =  \sum_{e>0}\frac{(xy)^{e/2}-(xy)^{-e/2}}{(xy)^{1/2}-(xy)^{-1/2}}\\
& \Bigg( \overline h(A^{[g-1]}) t^e+(t^{e}+t^{-e})\Bigg(\sum_{0<b<(g-1)/e} \overline h(A^{[g-1-be]})+\sum_{b=(g-1)/e}  b^2\Bigg)\Bigg).
\end{align*}
 \end{rem}

\section{The refined invariants  for surfaces with $K_S$ numerically trivial.}

In this section we prove \corref{numconj} and \thmref{higher}. Let 
$$x:=t+t^{-1}-y^{1/2}-y^{-1/2}=\frac{1}{t}(1-ty^{1/2})(1-t y^{-1/2})=-y^{-1/2}(1-y^{1/2}t)(1-y^{1/2}/t).$$
By \thmref{thm:genf}  it is enough to prove \corref{numconj} for K3 surfaces and Abelian surfaces. 
Let $S$ be a K3 surface or an Abelian surface with $Pic(S)=\Z$ generated by a line bundle  $L_g$  with $g(L_g)=g$.
These exist for any $g\ge 2$: for Abelian surfaces we can take a principally polarized abelian surface with a polarization of type
$(1,g-1)$. For K3 surfaces this result can for instance  be found in \cite[Prop.12]{Ch}, where the surfaces are defined as deformations of embeddings in $\P^g$ of hypersurfaces in $\P^1\times \P^2$ and in $\P^2$-bundles over $\P^1$.

We know $$\sum_{i}N^i_{[S,L_g],\chi(L_g)-1-k}(y)x^{g-i-1}=\Coeff_{q^{g-1}}\left[\X^k\K^{\chi(\oO_S)/2}\A^{1-\chi(\oO_S)/2}\right].$$
We know that $N_{[S,L_g],\chi(L_g)-1-k}^i=0$ for $i>\chi(L_g)-1-k=g-(k+2-\chi(\oO_S))$. 

This means $$\X^k\K^{\chi(\oO_S)/2}\A^{1-\chi(\oO_S)/2}\in x^{k+1-\chi(\oO_S)}\Q[x][q^{-1}][[q]]$$ and 
$$\sum_{g} N_{[S,L_g],\chi(L_g)-1-k}^{\chi(L_g)-1-k}(y)q^{g-1}=\frac{\X^k\K^{\chi(\oO_S)/2}\A^{1-\chi(\oO_S)/2}}{x^{k+1-\chi(\oO_S)}}\Big|_{t=y^{1/2}}.$$
By \thmref{thm:A} and the fact that $[d]_y|_{y=1}=d$, we get 
$$\frac{\A}{x}\Big|_{t=y^{1/2}}=\sum_{n>0,d>0}n^2d[d]^2_y q^{nd}=D\widetilde{DG}_2(y,q).$$
As $\H=D^{-1}\A$, we also see that $\frac{\H}{x}\big|_{t=y^{1/2}}=\widetilde{DG}_2(y,q).$
By \thmref{thm:K}   we get
$(x\K)|_{t=y^{1/2}}=\widetilde \Delta(y,q)$.
As $\X\in \H \cdot(1+\Q[y][[\H]])$, we find that 
$$\frac{\X}{x}\big|_{t=y^{1/2}}=\frac{\H}{x}\big|_{t=y^{1/2}}=\widetilde{DG}_2(y,q).$$
Substituting, we have proven  \corref{numconj}:
$$\sum_{g}N_{[S,L_g],\chi(L_g)-1-k}^{\chi(L_g)-1-k}(y)q^{g-1}=\frac{\widetilde{DG}_2(y,q)^k (D\widetilde {DG}_2(y,q))^{1-\chi(\oO_S)/2}}{\Delta(y,q)^{\chi(\oO_S)/2}}.$$

Now we prove \thmref{higher}. This proof (except the easy \eqref{z2ve}) is due to Don Zagier.
As before let $e^{z}:=y$, $e^{z_1}:=y_1:=ty^{1/2}$, $e^{z_2}:=y_2=y^{1/2}/t$. Note that $y_1y_2=y$.
We can rewrite the formula of \thmref{thm:K} as follows:
\[
\K = 
\frac{y^{-1/2}-y^{1/2}}{\Delta(q)}\frac{\theta'(0)^3}{\theta(y_1)\theta(y_2)\theta(y)}=\frac{1}{(y^{-1/2}-y^{1/2})\widetilde \Delta(y,q)} 
\frac{\theta'(0)\theta(y_1y_2)}{\theta(y_1)\theta(y_2)}.
\]
Let $\ve:=\frac{(y_1-1)(y_2-1)}{y-1}=\frac{x}{y^{-1/2}-y^{1/2}}.$ Then by \eqref{zagfun} we have
\begin{equation}\label{Kdel}
x\widetilde\Delta(y,q)\K  = 1-\ve \sum_{nd>0} \sgn(d) y_2^{n-d}y^dq^{nd}=
1-\ve \sum_{nd>0} \sgn(d)e^{-(d-n)z_2}y^dq^{nd}.
\end{equation}
\begin{REF} there was a $\sgn(d)$ missing: fixed.\end{REF}
Note that $\ve=\ve(z_2)=\frac{1}{y-1}(ye^{-z_2}-1)(e^{z_2}-1)$
is a power series in $\Q[\frac{1}{y-1}][[z_2]]$, starting with $z_2$.
Let $z_2(\ve)$ be the inverse series. Then the formula for $\K$  of the Theorem follows from \eqref{Kdel} together with the claim that
\begin{equation}\label{z2ve}
z_2(\ve)= \sum_{n\ge 0} \frac{s_{n}(y)}{(y^{1/2}-y^{-1/2})^n}\frac{\ve^{n+1}}{n+1}=-S(y,x).
\end{equation}
By the Lagrange inversion formula $\Coeff_{\ve^n} z_2=\frac{1}{n} \Res_{z_2=0}\ve^{-n}$, we need to see that 
$$(y-1)^n\Res_{z_2=0}\ve^{-n-1}=\Coeff_{x^n}\big((1+x)(1+xy)\big)^n.$$
We put $Y:=(y-1)$, $T=e^{z_2}-1$, then
\begin{align*}
(y-1)^n\Res_{z_2=0}\ve^{-n-1}&=\Coeff_{T^n}\frac{Y^{2n+1}(1+T)^n}{(Y-T)^{n+1}}=
\Coeff_{T^n}\frac{(1+TY)^n}{(1-T)^{n+1}}=\sum_{l=0}^n\binom{n}{l}\binom{2n-l}{n} Y^l.
\end{align*}
On the other hand
$$\Coeff_{x^nY^l}\big((1+x)(1+x(1+Y))\big)^n
=\sum_{k=0}^n\binom{n}{k}^2\binom{k}{l}=\sum_{k=0}^n\binom{n}{k}\binom{n-l}{n-k}\binom{n}{l}=
\binom{2n-l}{n}\binom{n}{l}.
$$
This establishes \eqref{z2ve} and thus the formula for $\K$. To prove the formula for $\A$, note that 
by the definition 
$A(y,q)=\sum_{nd>0} \sgn(d) n^2 (y^d-1) q^{nd}$. Thus we have by \thmref{thm:A} that
\begin{align*}(y^{1/2}&-y^{-1/2})\mathbb{A}  =A\left(y_1,q\right)+A\left(y_2,q \right)-A(y,q)=
\sum_{nd>0} \sgn(d)n^2\big(y_1^d+y_2^d-(y^d-1)\big) q^{nd}\\
&=\sum_{nd>0} \sgn(d)n^2\big(y_1^d-y_2^{-d}-(y^d-1)\big) q^{nd}=
\sum_{nd>0} \sgn(d)n^2 (e^{-z_2 d }-1)(y^d-1) q^{nd},\end{align*}
and use again \eqref{z2ve}. This finishes the proof of  the formula for $\mathbb{A}$. 
The formula for $\H$ now follows directly from the definition $\H=D^{-1}\mathbb{A}$. 
\begin{REF} Explain why there are line bundles of genus $g$ for all $g$ (in a suitable place).\end{REF}
\begin{LG} I made some changes: \corref{numconj} needs only to be proved for the case of K3 and Abelian surfaces, where the existence of the line bundles is clear.
\end{LG}
\begin{REF} Change citations like [Yos01], and [PT3]: fixed\end{REF}
\begin{LG} We should more generally check all the references to [GS]. I suggest that we take the numbering from the G and T version.
\end{LG}
\begin{REF} There are a number of displayed formulas where the puctuation is missing afterwards.
\end{REF}
\begin{LG} Went through it and fixed it.
\end{LG}
\begin{LG} Updated the references
\end{LG}

\end{document}